\documentclass[11pt,reqno,a4paper]{amsart}
\usepackage[english]{babel}
\usepackage[applemac]{inputenc}
\usepackage[T1]{fontenc}
\usepackage{palatino}
\usepackage{amsfonts}
\usepackage{amsmath}
\usepackage{amssymb}
\usepackage{amsthm}
\usepackage{graphicx}
\usepackage{wrapfig}
\usepackage{type1cm}
\usepackage[bf]{caption}
\usepackage{esint}
\usepackage{version}
\usepackage[colorlinks = true, citecolor = black, linkcolor = black, urlcolor = black, pdfstartview=FitH]{hyperref}
\usepackage{caption}
\usepackage{tikz}
\usepackage{bbding, wasysym}
\usepackage{textpos}

\DeclareCaptionType{figurecaption}[Figure][]

\newcommand{\R}{\mathbb{R}}

\newcommand{\Q}{\mathbb{Q}}
\newcommand{\N}{\mathbb{N}}
\newcommand{\W}{\mathbb{W}}
\newcommand{\w}{\mathbf{w}}
\renewcommand{\P}{\mathbb{P}}

\newcommand{\I}{\mathbb{I}}

\newcommand{\cM}{\mathcal{M}}

\newcommand{\cL}{\mathcal{L}}

\newcommand{\cR}{\mathcal{R}}

\newcommand{\cF}{\mathcal{F}}

\newcommand{\cT}{\mathcal{T}}

\newcommand{\cQ}{\mathcal{Q}}
\newcommand{\cP}{\mathcal{P}}
\newcommand{\cS}{\mathcal{S}}

\newcommand{\Z}{\mathbb{Z}}

\newcommand{\Lip}{\operatorname{Lip}}

\newcommand{\spt}{\operatorname{spt}}
\newcommand{\Tan}{\operatorname{Tan}}
\newcommand{\ZOOM}{\operatorname{ZOOM}}
\newcommand{\micro}{\operatorname{micro}}

\renewcommand{\emptyset}{\varnothing}
\renewcommand{\epsilon}{\varepsilon}
\renewcommand{\rho}{\varrho}
\renewcommand{\phi}{\varphi}

\newcommand{\meanbar}[1]{%
\setbox0 = \hbox{$#1 \int$}
\hbox to 0pt{%
\thinspace
\hskip 0.1\wd0
\raise 0.5\ht0
\hbox{%
\lower 0.5\dp0
\hbox{\rule{0.8\wd0}{2\linethickness}}
}%
\hss
}%
}

\theoremstyle{plain}

\newtheorem{theorem}{Theorem}
\newtheorem{lemma}{Lemma}
\newtheorem{proposition}{Proposition}

\newtheorem{corollary}{Corollary}

\newtheorem{problem}{Problem}

\newtheorem*{claim*}{Claim}

\theoremstyle{definition}
\newtheorem{definition}{Definition}

\newtheorem*{examples*}{Examples}
\newtheorem*{example*}{Example}
\newtheorem{remark}{Remark}

\newtheorem{notations}{Notations}

\newtheorem*{notations*}{Notations}
\newtheorem*{notation*}{Notation}

\numberwithin{equation}{section}
\numberwithin{theorem}{section}
\numberwithin{thm}{section}
\numberwithin{lemma}{section}
\numberwithin{proposition}{section}
\numberwithin{cor}{section}
\numberwithin{corollary}{section}
\numberwithin{claim}{section}
\numberwithin{definition}{section}
\numberwithin{example}{section}
\numberwithin{remark}{section}
\numberwithin{notations}{section}
\numberwithin{notation}{section}
\numberwithin{claim}{section}
\numberwithin{problem}{section}
\numberwithin{figure}{section}
\author{Tuomas Sahlsten}
\title{Tangent measures of typical measures}
\thanks{Research was supported by the Finnish Centre of Excellence in Analysis and Dynamics Research and Emil Aaltonen Foundation.}
\address{Department of Mathematics and Statistics, P.O.B. 68, FI-00014 University of Helsinki, Finland}
\email{tuomas.sahlsten@helsinki.fi}

\begin{document}

\begin{abstract}
We prove that for a typical Radon measure $\mu$ in $\R^d$, every non-zero Radon measure is a tangent measure of $\mu$ at $\mu$ almost every point. This was already shown by T. O'Neil in his PhD thesis from 1994, but we provide a different self-contained proof for this fact. Moreover, we show that this result is sharp: for any non-zero measure we construct a point in its support where the set of tangent measures does not contain all non-zero measures. We also study a concept similar to tangent measures on trees, micromeasures, and show an analogous typical property for them.
\end{abstract}

\maketitle

\section{Introduction}
\label{intro}

%A natural way to define topological genericity in complete metric spaces is the notion of typicality. 

If $X$ is a complete metric space, then we say that a subset of $X$ is \textit{meagre}, if it is a countable union of sets whose closure in $X$ has empty interior. A subset of $X$ is \textit{residual} if its complement is meagre. A property $\sf{P}$ of points $x \in X$ is satisfied for \textit{typical} $x \in X$ if the set
$$\{x \in X : x \text{ satisfies } \sf{P}\}$$
is residual. Recently, typical properties of measures have gained a lot of attention. For example, in the recent papers \cite{Bay12,BS10,Ols05,Ols07,Ols08} the $L^q$-dimensions and multifractal properties of typical measures were studied. This motivated us to study the tangential properties of typical measures. Our work is somewhat related to the papers by Buczolich and R\'ati \cite{Buc03, BR06} where the structure of the tangent sets of the graphs of typical continuous functions were studied.

In \cite{ONe95} O'Neil constructed a Radon measure $\mu$ in $\R^d$ with a very surprising property: for $\mu$ almost every $x \in \R^d$ the set of tangent measures $\Tan(\mu,x) = \cM \setminus \{0\}$, where $\cM$ is the space of all Radon measures. In his PhD thesis \cite{ONe94} O'Neil also extended this result by showing that such a property of measures is actually typical:

\begin{theorem}
\label{main1}
A typical $\mu \in \cM$ satisfies $\Tan(\mu,x) = \cM \setminus \{0\}$ at $\mu$ almost every $x \in \R^d$.
\end{theorem}

In this paper, we provide a different self-contained proof for Theorem \ref{main1}. O'Neil's original proof relied on a special property of the measure $\mu$ constructed in \cite{ONe95}, but here we do not require O'Neil's measure in our approach.

As a direct consequence of Theorem \ref{main1} we notice that a typical measure $\mu$ is \textit{non-doubling} in $\R^d$, that is, the pointwise doubling condition fails $\mu$ almost everywhere, see Section \ref{application} for details. We also study the sharpness of Theorem \ref{main1}, that is, whether the property $\Tan(\mu,x) = \cM \setminus \{0\}$ can be extended to hold at \textit{every} point $x \in \spt \mu$ for a typical $\mu$. However, such an extension is not possible since for any given $\mu \in \cM$ with non-empty support $\spt \mu$ we find a point $x \in \spt \mu$ such that $\Tan(\mu,x) \neq \cM \setminus \{0\}$; see Section \ref{sharpness}.

Furthermore, we also take a quick look at a similar concept to tangent measures, the so called micromeasures, which provide a symbolic way to define ``tangent measures'' of a measure in a tree. We consider the set of all Borel probability measures $\cP$ on the tree $I^\N$, where $I$ is some finite set, and prove an analogous result for micromeasures that we had for tangent measures: for a typical $\mu \in \cP$ the set of micromeasures $\micro(\mu,x) = \cP$ at \textit{every} point $x \in I^\N$, see Section \ref{secmicro} for details. Finally, in Section \ref{problems} we exhibit some questions analogous to Theorem \ref{main1} about the micromeasure distributions of typical measures and the tangent measures of measures that are generic in the sense of prevalence instead of typicality.

\begin{remark}
The main result Theorem \ref{main1} was initially proved independently without any knowledge of the existence of O'Neil's proof in his PhD thesis \cite{ONe94} from 1994, as the same result there was not published in a journal. This was only later brought to the author's attention by O'Neil after the manuscript was submitted to arXiv article repository on 19th of March 2012 (\texttt{http://arxiv.org/abs/1203.4221v1}).
\end{remark}

\section{Preliminaries}
\label{pre}

Throughout this paper, we keep the dimension $d\in\N$ of the ambient space $\R^d$ fixed. A \textit{measure} is a Radon-measure on $\R^d$, and their collection is denoted by $\cM$. We equip $\cM$ with the weak topology that is characterized by the convergence:  if $\mu_i,\mu \in \cM$ we say that $\mu_i \to \mu$, as $i \to \infty$, if
$$\int \phi \, d\mu_i \longrightarrow \int \phi \, d\mu, \quad \text{as } i \to \infty,$$
for every compactly supported and continuous $\phi : \R^d \to \R$. In metric spaces, the \textit{open-} and \textit{closed balls} of center $x$ and radius $r$ are denoted by $U(x,r)$ and $B(x,r)$. When $\mu \in \cM$, the \textit{support} of $\mu$ is the set $\spt \mu = \{x \in \R^d : \mu(B(x,r)) > 0 \text{ for any } r > 0\}$. When $x \in \R^d$ and $r > 0$, let $T_{x,r} : \R^d \to \R^d$ be the affine homothety that maps $B(x,r)$ onto $B(0,1)$, that is, $T_{x,r}(y) = (y-x)/r$, $y \in \R^d$. Given $\mu \in \cM$, we write
$$T_{x,r \sharp} \mu(A) = \mu(rA+x), \quad A \subset \R^d,$$
that is, the push-forward of $\mu$ under the map $T_{x,r}$. When $c > 0$ we also write 
$$\cT_{x,r,c}(\mu) = cT_{x,r \sharp} \mu,$$
which induces a map $\cT_{x,r,c} : \cM \to \cM$. In the case $r = 1$, we just have $T_{x,1 \sharp} \mu =: \mu-x$. The following notion was introduced by D. Preiss in \cite{Pre87}:

\begin{definition}[Tangent measures]
\label{tangent}
A measure $\nu \in \cM \setminus \{0\}$ is a \textit{tangent measure} of $\mu \in \cM$ at $x \in \R^d$ if there exist $r_i \searrow 0$ and $c_i > 0$ such that
$$\cT_{x,r_i,c_i}(\mu) =c_i T_{x,r_i \sharp} \mu \longrightarrow \nu, \quad \text{ as } i \to \infty.$$
The set of all tangent measures of $\mu$ at $x$ is denoted by $\Tan(\mu,x)$, which is a closed subset of $\cM \setminus \{0\}$.
\end{definition}

Next we introduce some key notations for the proof of Theorem \ref{main1}:

\begin{notations}[Cube filtrations and weighted cubes] 
\label{cubeweight}
Fix $a \in \Z$.
\begin{itemize}
\item[(1)] Write
$$\I_a := [-3^{a}/2,3^{a}/2)^d$$
and for notational simplicity let $\I := \I_0$. Moreover, if $\epsilon > 0$ is fixed we let the $\epsilon$-\textit{expansions} and $\epsilon$-\textit{contractions} of the cube $\I_a$ to be the sets
$$\I^+_{a,\epsilon} =  [-3^a/2-\epsilon,3^a/2+\epsilon)^d\quad \text{and} \quad \I^-_{a,\epsilon} = [-3^a/2+\epsilon,3^a/2-\epsilon)^d.$$
\item[(2)] Suppose $a > 0$. Fix $k \in \Z$. Let $\cQ^k_a$ be the collection of all $3^a$-\textit{adic cubes} $Q$ of side-length $\ell(Q) = 3^{-ak}$ such that the unit cube $\I \in \cQ^0_a$, where $k$ is the \textit{generation} of the cubes in $\cQ^k_a$. Write $\cQ_a = \bigcup_{k \in \Z} \cQ^k_a$. If $Q \in \cQ_a^k$, we let $x(Q)$ be the \textit{central point} of $Q$. Moreover, let $Q_c$ be the \textit{central cube} amongst all the cubes $Q' \subset Q$, $Q' \in \cQ^{k+2}_a$, that is, $Q_c \in \cQ^{k+2}_a$ is uniquely determined by the requirement $x(Q_c) = x(Q)$. Notice that the central cube $Q_c$ is two generations younger than $Q$.
\item[(3)] If $Q \in \cQ^k_a$, let $Q^j \in \cQ_a^k$, $j = 2,\dots,3^{d}$, be all the \textit{neighbouring cubes} of $Q$ and write $Q^1 = Q$. Let
$$\W = \{\w = (w_1,w_2,\dots,w_{3^d}) : w_1 = 1 \text{ and } w_j > 0, j =2,\dots,3^d\}.$$
When $\nu \in \cM$ and $\w \in \W$ are fixed we will denote for all $j =1,\dots,3^d$ the $w_j$-\textit{weighted duplication} of the restriction $\nu_a := \nu \llcorner \I_a$ to the neighbouring cube $\I_a^j$ by:
$$\nu^{w_j,j}_a = w_j [\nu_a+x(\I_a^j)],$$
which is the same measure as $w_j T_{-x(\I_a^j),1 \sharp} \nu_a = \cT_{-x(\I_a^j),1,w_j}(\nu_a)$. Then write
$$\nu_{a}^\w = \sum_{j = 1}^{3^d} \nu^{w_j,j}_a.$$
Notice that $\nu_a^{\w} \llcorner \I_a = \nu_a$ for any $\w \in \W$. See Figure \ref{figEpsilon} for some intuition of using this notation.
\end{itemize}
\end{notations}

\begin{definition}[Metric on measures]
\label{metricmeasure} Fix $a \in \N$. Let $\cL(a)$ be the set of all Lipschitz functions $\phi : \R^d \to [0,\infty)$ with Lipschitz-constant $\Lip \phi \leq 1$ and support $\spt \phi \subset \I_a$. For $\mu,\nu \in \cM$ we write:
$$F_a(\mu,\nu) = \sup_{\phi \in \cL(a)} \Big|\int \phi \, d\mu- \int \phi \, d\nu\Big|$$
and
$$d(\mu,\nu) = \sum_{a = 1}^\infty 2^{-a} \min\{1,F_a(\mu,\nu)\}.$$
Then $(\cM,d)$ is a complete separable metric space, and the topology induced by $d$ agrees with the weak convergence. Note that here we abuse notation: $d$ also refers to the dimension of the ambient space $\R^d$.
\end{definition}

\begin{remark}
\label{remarkmetric}
\begin{itemize}
\item[(1)]
A similar metric of measures was used in \cite[Remark 14.15]{Mat95} with the difference that the closed ball $B(0,a)$ is used instead of $\I_a$ in the definition of $\cL(a)$. This changes the value of the metric $d$, but still all the properties of $d$ and $F_a$ given in \cite{Mat95} are satisfied. Especially, we have the following characterization of weak convergence: let $\mu_i,\mu \in \cM$, $i \in \N$. Then
$$\mu_i \to \mu \quad \Leftrightarrow \quad d(\mu_i,\mu) \to 0 \quad \Leftrightarrow \quad F_b(\mu_i,\mu) \to 0, i \to \infty \text{ for all } b \in \N,$$
see the proof of \cite[Lemma 14.13]{Mat95}.
\item[(2)] For a fixed $a \in \N$ we let the open ball with respect to the metric $F_a$ be:
$$U_a(\nu,\epsilon) = \{\mu \in \cM : F_a(\mu,\nu) < \epsilon\}.$$
It follows immediately that this set is also open with respect to the metric $d$. 
\end{itemize}
\end{remark}

\section{Proof of the main result}
\label{proofofmain}

In order to prove Theorem \ref{main1}, it is enough to construct a subset
$$\cR \subset \{\mu \in \cM : \Tan(\mu,x) = \cM \setminus\{0\} \text{ for } \mu \text{ a.e. } x \in \R^d\},$$
which is a countable intersection of open dense sets in $\cM$. We will now fix a number of parameters required to define such a set.

If $Q \in \cQ_1$, we let $\cL_Q$ be the normalized $d$-dimensional Lebesgue measure supported on $Q$, that is, $\cL_Q = \cL^d(Q)^{-1}\cL^d \llcorner Q$. Write
$$\cS = \Big\{\cL^d \llcorner (\R^d \setminus \I_n) + \sum_{Q} q_Q \cL_Q : q_Q > 0, q_Q \in \Q \text{ where } Q \in \cQ_1^n, Q \subset \I_n, n \in \N \Big\}.$$
Then $\cS \subset \cM$ is countable and dense in $(\cM,d)$. We especially need the following properties of measures $\nu \in \cS$ in our proof:
\begin{itemize}
\item[(1)] $\spt \nu = \R^d$;
\item[(2)] $\nu(\partial \I_a) = 0$ for every $a \in \Z$.
\end{itemize} 

\begin{definition}[Choices of $\beta_a$, $\epsilon_a$ and $\epsilon_a^\w$]\label{choices}
Let $\nu \in \cS$. Choose any sequence $\beta_a = \beta_a(\nu) \searrow 0$ with $\beta_a < 3^{-a}\nu(\I_{-a})^{-1}/4$ for any $a \in \N$. If $a \in \N$, we write:
$$\epsilon_a := \beta_a \nu(\I^d_{-a}) \in (0,3^{-a}/4).$$
Fix $a \in \N$ and $\w \in \W$. Choose any number $\epsilon \in (0,\epsilon_a)$ such that
\begin{align}
\label{alphachoice}\max\Big\{\nu(\I^+_{-a,\epsilon} \setminus \I^-_{-a,\epsilon}),\nu_a^{\w}(\I^+_{a,\epsilon} \setminus \I^-_{a,\epsilon})\Big\} < \epsilon_a.
\end{align}
We denote $\epsilon_a^\w := \epsilon$ to emphasize the dependence on $a$ and $\w$ for the choice of $\epsilon$. All this is possible because $\nu \in \cS$ and thus $\nu(\partial \I_{-a}) = 0 = \nu(\partial \I_{a})$. Indeed, this yields
$$\lim_{\epsilon \to 0} \nu(\I^+_{-a,\epsilon} \setminus \I^-_{-a,\epsilon}) = 0 = \lim_{\epsilon \to 0} \nu_a^{\w}(\I^+_{a,\epsilon} \setminus \I^-_{a,\epsilon}),$$
recall the Notation \ref{cubeweight}(3).
\end{definition}

\begin{definition}[The set $\cR$]
\label{setR}
If $a \in \N$ and $k \in \N$ we let
$$r_a^k = 3^{-(k+1)a}/2.$$
This number is half the side-length of a cube in $\cQ^{k+1}_a$. We are now planning to construct a countable intersection $\cR$ of open and dense sets. For each measure $\nu \in \cS$, parameter $a \in \N$ and generation $n \in \N$ we associate a set $\cR_{\nu,a,n} \subset \cM$ as follows. This subset consists of all measures $\mu \in \cM$ with the property that for a deep enough generation $k \geq n$, and for all cubes $Q \in \cQ^k_a$, $Q \subset \I_a$, there exists a normalization constant $c > 0$ and a weight vector $\w \in \W$ such that the blow-up $\cT_{x(Q),r_{a}^k,c}(\mu) = cT_{x(Q),r_{a}^k\sharp}\mu$ is $\epsilon_a\epsilon_a^\w$-close (in the $F_{a+1}$-distance) to the $\w$-weighted measure $\nu_a^\w$, see also Figure \ref{figMapT}.
In other words
$$\cR_{\nu,a,n} := \bigcup_{k \geq n} \bigcap_{Q \in \cQ^k_a \atop Q \subset \I_a} \bigcup_{c > 0} \bigcup_{\w \in \W} \cT_{x(Q),r_a^k,c}^{-1} U_{a+1}(\nu_{a}^\w,\epsilon_{a}\epsilon_{a}^{\w}).$$
There are only countably many $\cR_{\nu,a,n}$, $\nu \in \cS$, $a \in \N$ and $n \in \N$, so
$$\cR := \bigcap_{\nu \in \cS} \bigcap_{a \in \N}\bigcap_{n \in \N} \cR_{\nu,a,n}$$
is a countable intersection. See the outline of the proof in below for more heuristics on the choice of the parameters and the set $\cR$.
\end{definition}

\begin{figure}[h]
\begin{center}
\includegraphics[scale=0.55]{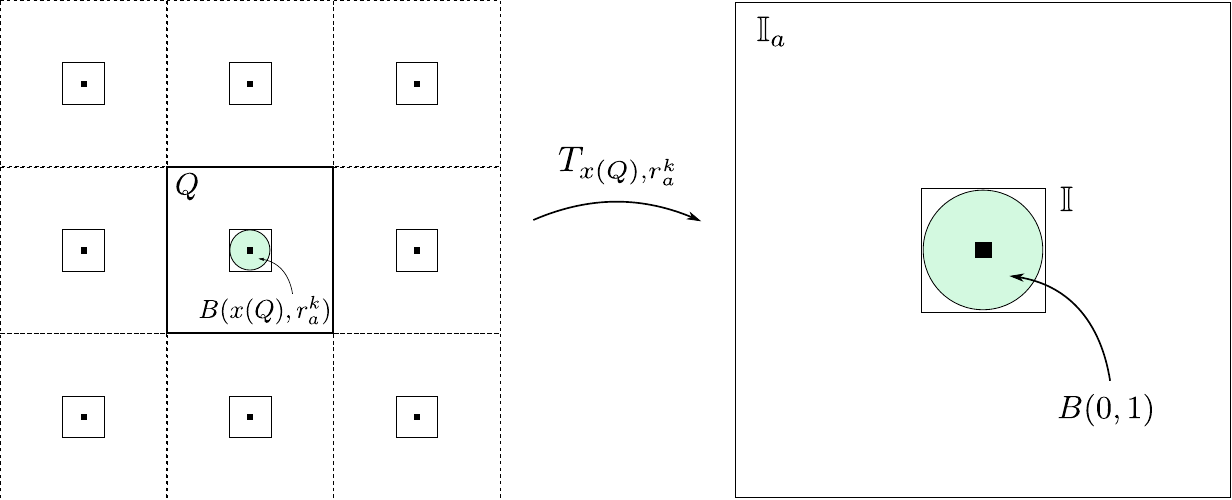}
\end{center}
\caption{The map $T_{x(Q),r_a^k}$ used in the definition of $\cR_{\nu,a,n}$ maps $Q$ onto $\I_a$ and $Q_c$ onto $\I_{-a}$ (the small black cube on the right-hand side), respectively.}
\label{figMapT}
\end{figure}

\textbf{Outline of the proof.} 
Since $\cS$ is dense in $\cM$ and $\Tan(\mu,x)$ is always closed in $\cM \setminus \{0\}$, we only need to verify for each $\nu \in \cS$ and $\mu \in \cR$ that $\nu \in \Tan(\mu,x)$ for $\mu$ almost every $x \in \R^d$. The set $\cR$ has the property that when $\nu \in \cS$ and $a \in \N$ are fixed we can find arbitrarily large generations $k \in \N$ such that the measure $\mu \in \cR$ will look in all cubes $Q \in \cQ_a^k$ like a small translate of $\nu_a^\w$ when we blow-up with respect to any point $x \in Q_c$, recall Notations \ref{cubeweight}(2). Since the relative size of the central cube $Q_c$ becomes very small compared to their ancestor in $\cQ_a^{k+1}$ when $a$ is large (in the factor of $3^{-a}$), the translates of $\nu_a^\w$ tend to look like $\nu$ since $\nu_a^\w$ restricted to $\I_a$ is $\nu \llcorner \I_a$. Here we need to use the measures $\nu_{a}^\w$ and the weights $\w \in \W$ in order to make $\cR_{\nu,a,n}$ dense in $\cM$.

Hence we should try to somehow cover $\mu$ almost every point of $\R^d$ with such nice cubes. What we will first do is that we fix some numbers $a,b \in \N$, and then invoke the definition of $\cR$ to find infinitely many generations $k$ such that the central cubes $Q_c$ of the cubes $Q \in \cQ_a^k$ cover some portion of some large reference cube $\I_b$ with respect to the measure $\mu$. However, verifying this produces some of the trickier parts of the proof. To this end we need the following generalization of the Borel-Cantelli Lemma (see for example \cite{Ren70}), where the condition on independence is replaced with a more quantitative statement:

\begin{lemma}
\label{borelcantelli}
Let $(\Omega,\cF,\P)$ be a probability space and $A_n \in \cF$, $n \in \N$, such that $\sum_{n = 1}^\infty \P(A_n) = \infty$. Then
$$\P\Big(\limsup_{n \to \infty} A_n\Big) \geq \limsup_{N \to \infty}\frac{\Big(\sum_{n = 1}^N \P(A_n)\Big)^2}{\sum_{n =1}^N \sum_{l = 1}^N \P(A_n \cap A_l)}.$$
\end{lemma}

In the proof of Theorem \ref{main1}, the events $A_n$ are exactly the unions $A_{a,n}^b$, $n \in \N$, of all $3^a$-adic central cubes $Q_c$ of certain generation $k_n$ cubes $Q$ in some large reference cube $\I_b$. Moreover, $\P$ is the normalization of $\mu \llcorner \I_b$ such that we have $F_{a+1}(cT_{x(Q),r_a^k\sharp} \mu,\nu_a^\w) < \epsilon_a\epsilon_a^\w$ for some $c = c(Q) > 0$ and $\w = \w(Q) \in \W$. We need the more general form of Borel-Cantelli's lemma here since our events $A_{a,n}^b$, $n \in \N$, are not in general $\P$-independent, but when $n \to \infty$, we can say something about their pairwise correlations.

In order to apply Lemma \ref{borelcantelli} we need to compare the measures $\mu(Q)$ and $\mu(Q_c)$ to each other using the comparison of $\nu$ measures of the reference cubes $\I_a = T_{x(Q),r_a^k}(Q)$ and $\I_{-a} = T_{x(Q),r_a^k}(Q_c)$, which is made possible by the knowledge of $F_{a+1}(cT_{x(Q),r_a^k\sharp} \mu,\nu_a^\w)$. In this way we gain the right measures for the sets $A_{a,n}^b$ and $A_{a,n}^b \cap A_{a,l}^b$.

However, when we do the $\mu$ measure comparison, we end up having some error terms coming out from the $\nu_a^\w$ measures of the buffer zones $\I^+_{a,\epsilon} \setminus \I^-_{a,\epsilon}$ and $\I^+_{-a,\epsilon} \setminus \I^-_{-a,\epsilon}$. However, by the choices we made in Definition \ref{choices}, these errors are at most of the size $\epsilon_a$, which is independent of generations $n$. Then we apply Lemma \ref{borelcantelli} to see that the $\mu$ measure of $A_a^b = \limsup_n A_{a,n}^b$ is nearly the same as $\mu(\I_b)$, and how near will depend on the numbers $\beta_a$ that arise from the errors $\epsilon_a$. Then it turns out that the set $A^b = \limsup_a A^b_a$ covers $\mu$ almost every point of $\I_b$, since as $a \to \infty$ the numbers $\beta_a \searrow 0$ by their choice. This way $\mu$ almost every point of the space $\R^d$ can be covered by the union of such sets $A^b$, $b \in \N$.

\begin{figure}[h]
\begin{center}
\includegraphics[scale=0.55]{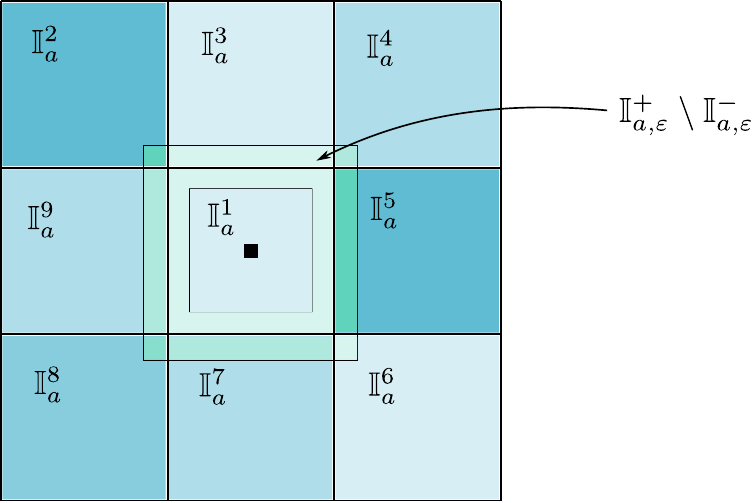}
\end{center}
\caption{The cube $\I_a^1 = \I_a$ and its neighbouring cubes $\I_a^j$, $j = 2,\dots,3^2$. We have weighted the cubes $\I_a^j$ with weights $w_j$, where the shade of the cube tells us how big the value of the weight $w_j$ is. This illustrates the measure $\nu_a^\w$: on $\I_a^j$ it equals to $w_j \nu_a$ translated to $\I_a^j$. We choose $\epsilon = \epsilon_a^\w$ such that the buffer zone $\I^+_{a,\epsilon} \setminus \I^-_{a,\epsilon}$ in the picture has $\nu_a^\w$ measure less than a fixed number $\epsilon_a > 0$. The bigger the weights $w_j$ are, the smaller $\epsilon$ we have to choose. The small black cube in the picture is $\I_{-a}$, and we want to choose $\epsilon$ to be small enough that even the $\nu$ measure of the small buffer zone $\I^+_{-a,\epsilon} \setminus \I^-_{-a,\epsilon}$ is less than $\epsilon_a$.}
\label{figEpsilon}
\end{figure}

\begin{lemma}
\label{opendense}
$\cR_{\nu,a,n}$ is open and dense in $\cM$.
\end{lemma}

\begin{proof}
(1) Let us first prove that $\cR_{\nu,a,n}$ is open in $\cM$. Fix $x \in \R^d$, $r > 0$ and $c > 0$. We will now show that $\cT := \cT_{x,r,c} : \cM \to \cM$ is continuous. This is enough for our claim since the balls $U_{a+1}$ in the definition of $\cR_{\nu,a,n}$ are also open in $(\cM,d)$ and the intersection in the definition of $\cR_{\nu,a,n}$ has a finite index set $\{Q \in \cQ^k_a : Q \subset \I_a\}$.  Suppose $\mu_i,\mu \in \cM$ are chosen such that $\mu_i \to \mu$. We need to verify for any fixed compactly supported continuous $\phi : \R^d \to \R$ we have
$$\int \phi\, d\cT(\mu_i) \to \int \phi\, d\cT(\mu), \quad \text{as } i \to \infty.$$
Hence fix a continuous and compactly supported $\phi : \R^d \to \R$, which makes $\phi \circ T_{x,r}$ also continuous and compactly supported. Since $\mu_i \to \mu$, we have
\begin{align*}
\Big|\int \phi \,  d\cT(\mu_i) - \int \phi\, d\cT(\mu)\Big| = c \Big|\int \phi \circ T_{x,r} \,  d\mu_i- \int \phi \circ T_{x,r} \, d\mu\Big| \longrightarrow 0,
\end{align*}
as $i \to \infty$. Hence $\cT(\mu_i) \to \cT(\mu)$ as $i \to \infty$, so $\cT$ is continuous like we claimed.

(2) Here we prove that $\cR_{\nu,a,n}$ is dense in $\cM$. Let $\mu \in \cM$ be a measure with $\spt \mu = \R^d$. For $k \in \N$ we write
$$\mu_k = \sum_{Q \in \cQ^k_a} \mu(Q)\nu^{Q}$$
where
$$\nu^{Q} := \cT^{-1}_{x(Q),r_a^k,\nu(\I_{a})}(\nu_{a}), \quad Q \in \cQ^k_a.$$
Fix $Q \in \cQ^k_a$. Notice that $\nu^{Q}(Q) = 1$ and $\spt \nu^{Q} = Q$. Since $\spt \mu = \R^d$ each of the numbers $\nu(\I_{a})/\mu(Q^j)$, $j = 1,\dots,3^d$, is well-defined, where $Q^j$ is the $j$th neighbouring cube of $Q$, recall Notations \ref{cubeweight}(3). Write
\begin{align}
\label{const}c(Q) := \frac{\nu(\I_{a})}{\mu(Q)} > 0.
\end{align}
Moreover, define weights $w_{1} = 1$ and
 $$w_j = c(Q) \cdot \frac{\mu(Q^j)}{\nu(\I_a)}, \quad j = 2,\dots,3^{d}.$$
Then
 $$\w = (w_1,\dots,w_{3^d}) \in \W.$$
 Let $\phi \in \cL(a+1)$. Since $\spt \phi \subset \I_{a+1} = \bigcup_{j = 1}^{3^d} \I_a^j$, we have by the choice \eqref{const} and the definition of weights $w_j$ that
$$\int \phi \, d\cT_{x(Q),r_a^k,c(Q)}(\mu_k) = \int \phi \, d\nu_a^\w.$$
Since $\phi \in \cL(a+1)$ is arbitrary, we have especially:
$$F_{a+1}(\cT_{x(Q),r_a^k,c(Q)}(\mu_k),\nu_a^\w) = 0,$$
so in particular
$$\mu_k \in \cT_{x(Q),r_a^k,c(Q)}^{-1} U_{a+1}(\nu_a^\w,\epsilon_a\epsilon_a^{\w}).$$ 
Since this is true for every $k \in \N$ and $Q \in \cQ^k_a$, we have $\mu_k \in \cR_{\nu,a,n}$ whenever $k \geq n$. With this in mind, let us finally verify
$$\mu_k \to \mu, \quad \text{as } k \to \infty.$$
Let $\phi : \R^d \to \R$ be a compactly supported continuous function. Then we may fix $b \geq a$ such that $\spt \phi \subset \I_{b}$. Fix $\epsilon > 0$. Since $\phi$ is uniformly continuous, we can choose $k_\epsilon \in \N$ such that for every $k \geq k_\epsilon$ and $Q \in \cQ^k_a$, we have: $|\phi(y)-\phi(x)| < \epsilon$ whenever $y,x \in Q$. On the other hand, $\nu^{Q}(Q) = 1$ for every $Q \in \cQ_a^k$, so we have:
\begin{align*}
\Big|\int \phi \, d\mu_k - \int \phi \, d\mu\Big| &=\Big|\sum_{Q \in \cQ^k_a \atop Q \subset \I_b} \Big(\int_{Q}[\phi-\phi(x(Q))] \, d\mu_k + \int_{Q}[\phi(x(Q))-\phi] \, d\mu\Big) \Big| \\
& \leq \sum_{Q \in \cQ^k_a \atop Q \subset \I_b} \int_{Q}|\phi-\phi(x(Q))| \, d\mu_k + \sum_{Q \in \cQ^k_a \atop Q \subset \I_b}\int_{Q}|\phi(x(Q))-\phi| \, d\mu \\
& \leq \sum_{Q \in \cQ^k_a \atop Q \subset \I_{b}} \epsilon \mu_k(Q)+\sum_{Q \in \cQ^k_a \atop Q \subset \I_{b}} \epsilon\mu(Q)= 2\epsilon \sum_{Q \in \cQ^k_a \atop Q \subset \I_{b}} \mu(Q) = 2\epsilon\mu(\I_b),
\end{align*}
since $\mu_k(Q) = \mu(Q) \nu^Q(Q) = \mu(Q)$ for any $Q \in \cQ^k_a$. Hence $\mu_k \to \mu$, as $k \to \infty$. 

Measures $\mu$ with $\spt \mu = \R^d$ are dense in $\cM$. Hence if $\mu' \in \cM$ is any measure, for any $\epsilon > 0$ we can choose $\mu \in \cM$ with $d(\mu',\mu) < \epsilon/2$ and $\spt \mu = \R^d$. Then just choose $k \geq n$ so large that $d(\mu_k,\mu) < \epsilon/2$, which gives $d(\mu_k,\mu') < \epsilon$. The measure $\mu_k \in \cR_{\nu,a,n}$ so $\cR_{\nu,a,n}$ is dense.
\end{proof}

\begin{lemma}
\label{tangentsforR}
If $\mu \in \cR$, then $\Tan(\mu,x) = \cM \setminus \{0\}$ for $\mu$ almost every $x \in \R^d$.
\end{lemma}

\begin{proof} Fix $\mu \in \cR$. Since $\Tan(\mu,x)$ is closed in $\cM \setminus \{0\}$ and $\cS$ is dense in $\cM$, it is enough to show that $\cS \subset \Tan(\mu,x)$ for $\mu$ almost every $x \in \R^d$. 

Fix $\nu \in \cS$ and $a \in \N$. Since $\mu \in \cR$, we can choose for every $n \in \N$ an index $k := k_n \geq n$ such that for each $Q \in \cQ^k_a$, $Q \subset \I_a$, there are numbers $c = c(Q) > 0$ and weights $\w = \w(Q) \in \W$ such that
$$\mu \in \cT_{x(Q),r_a^k,c}^{-1} U_{a+1}(\nu_{a}^{\w},\epsilon_a\epsilon_a^\w).$$
Write
$$\mu_{Q} = \cT_{x(Q),r_a^k,c}(\mu) = cT_{x(Q),r_a^k\sharp} \mu.$$
Especially this measure satisfies
\begin{align}
\label{distbound}
F_{a+1}(\mu_Q,\nu_a^{\w}) < \epsilon_a \epsilon_a^\w.
\end{align}
Consider the sets
$$A_{a,n} = \bigcup_{Q \in \cQ^k_a \atop Q\subset \I_a} Q_c, \quad A_a = \limsup_{n \to \infty} A_{a,n}, \quad \text{and} \quad A = \limsup_{a \to \infty} A_a,$$
keeping in mind that $k = k_n$ and $Q_c \in \cQ^{k+2}$ is the central cube of $Q$, recall Notation \ref{cubeweight}(2). Let us first show that
\begin{align}
\label{limsup}
\mu(\R^d \setminus A) = 0.
\end{align}
We may assume that $\mu(\R^d) > 0$ since otherwise \eqref{limsup} is trivial. Then we may choose $b_0 \in \N$ such that $\mu(\I_{b_0}) > 0$. Fix $b \geq b_0$. Then
$$\P = \mu(\I_b)^{-1}\mu \llcorner \I_b$$
is a well-defined probability measure on $\R^d$ and $\spt \P \subset \I_b$. Write
$$A_{a,n}^b = A_{a,n} \cap \I_b, \quad A_a^b = \limsup_{n \to \infty} A_{a,n}^b, \quad \text{and} \quad A^b = \limsup_{a \to \infty} A_a^b.$$
We will now show that for any $a \geq b$ we have:
\begin{align}
\label{downmeasure}
\P(A_a^b)\geq \Big(\frac{1-2\beta_a}{1+2\beta_a}\Big)^4,
\end{align}
where $\beta_a$ is the number from Definition \ref{choices}. Let us first estimate the measure of $A_{a,n}^b$ in the case of $a \geq b$. When $Q \in \cQ^k_a$ is fixed, we will write for notational simplicity $\epsilon := \epsilon_a^\w$, and
$$\I^{+}_{a} := \I^{+}_{a,\epsilon}, \quad \I^{-}_{a} := \I^{-}_{a,\epsilon}, \quad \I^{+}_{-a} := \I^{+}_{-a,\epsilon}, \quad \text{and} \quad \I^{-}_{-a} := \I^{-}_{-a,\epsilon},$$
recall the definition of $\epsilon$-extensions from Notations \ref{cubeweight}(1), but keep in mind that these cubes depends on the cube $Q$. Choose $\phi^+_a,\phi^-_a,\psi^+_a,\psi^-_a \in \cL(a+1)$ as follows:
\begin{tabbing}
\quad \= (1) \quad \= $0 \leq \phi^+_a \leq \epsilon \chi_{\I^+_a}$, \quad \= $\phi^+ | \I_a \equiv \epsilon$, \quad \= and \quad \= $0 \leq \phi^-_a \leq \epsilon \chi_{\I_a}$, \quad \= $\phi^-_a | \I^-_a \equiv \epsilon$; \\
\> (2) \> $0 \leq \psi^+_a \leq \epsilon \chi_{\I^+_{-a}}$, \> $\psi^+_a | \I_{-a} \equiv \epsilon$, \> and \> $0 \leq \psi^-_a \leq \epsilon \chi_{\I_{-a}}$, \> $\psi^-_a | \I^-_{-a} \equiv \epsilon$.
\end{tabbing}
This is possible, since we have chosen $\epsilon = \epsilon_a^\w < \epsilon_a < 3^{-a}/4 = \ell(\I^d_{-a})/4$ so $\I^+_a = \I^{+}_{a,\epsilon} \subset \I_{a+1}$ and even in the small cube $\I^-_{-a}$ there is room to extend piecewise $1$-linearly the characteristic funtion of $\I^-_{-a}$ times $\epsilon$ to $\I_{-a}$. We will now prove:
\begin{align}
\label{help1}
|\mu_{Q}(\I_a) - \nu(\I_a)| < 2\beta_a\nu(\I_{-a});
\end{align}
and
\begin{align}
\label{help2}
|\mu_{Q}(\I_{-a}) - \nu(\I_{-a})| < 2\beta_a\nu(\I_{-a}).
\end{align}
Since $w_1 = 1$ (the weight of $\I_a(1) = \I_a$ is $1$), we always have:
$$\nu_a^\w \llcorner \I_{a} = \nu_a = \nu \llcorner \I_a.$$
Now recall \eqref{alphachoice}. If $\mu_{Q}(\I_a) > \nu(\I_a)$ we have by the estimate \eqref{distbound} that
\begin{align*}
\mu_{Q}(\I_a) - \nu(\I_a) &\leq \frac{1}{\epsilon} \int \phi^+_a \, d\mu_{Q} - \nu_a^\w(\I^-_a) \\
&\leq \frac{1}{\epsilon} \Big|\int \phi^+_a \, d\mu_{Q} - \int \phi^+_a \, d\nu^\w_a\Big|  +  \frac{1}{\epsilon}\int \phi^+_a \, d\nu^\w_a - \nu_a^\w(\I^-_a)\\
& \leq \frac{F_{a+1}(\mu_{Q},\nu_a^\w)}{\epsilon}+\nu_a^\w(\I^+_a \setminus \I^-_a) \leq \epsilon_a + \nu_a^\w(\I^+_a \setminus \I^-_a) \\
& < 2\beta_a \nu(\I_{-a}),
\end{align*}
and if $\mu_{Q}(\I_a) \leq \nu(\I_a)$, we have similarly
\begin{align*}
\nu(\I_a) - \mu_{Q}(\I_a) &\leq \nu_a^\w(\I^+_a) - \frac{1}{\epsilon} \int \phi^-_a \, d\mu_{Q} \\
&\leq \nu_a^\w(\I^+_a) - \frac{1}{\epsilon} \int \phi^-_a \, d\nu^\w_a+\frac{1}{\epsilon} \Big|\int \phi^-_a \, d\nu^\w_a - \int \phi^-_a \, d\mu_{Q}\Big|\\
& \leq \nu_a^\w(\I^+_a \setminus \I^-_a)+\frac{F_{a+1}(\mu_{Q},\nu_a^\w)}{\epsilon} \leq \nu_m^\w(\I^+_a \setminus \I^-_a)+\epsilon_a \\
& < 2\beta_a \nu(\I^d_{-a}).
\end{align*}
Hence \eqref{help1} holds. If we invoke again the estimates \eqref{alphachoice} and \eqref{distbound}, and now additionally the properties $\nu_a^{\w} \llcorner \I_a = \nu \llcorner \I_a$ and $\I^+_{-a} \subset \I_a$, and the choices of  $\psi^{\pm}_a$ we can prove \eqref{help2} with a symmetric argument as above. Write
$$\rho_a = \frac{1+2\beta_a}{1-2\beta_a} \quad \text{and} \quad p_a = \frac{\nu(\I_{-a})}{\nu(\I_{a})}.$$
Since $\nu(\I_{-a}) \leq \nu(\I_{a})$, the estimates \eqref{help1} and \eqref{help2}, $T_{x(Q),r_a^k}(Q) = \I_a$ and $T_{x(Q),r_a^k}(Q_c) = \I_{-a}$ imply
\begin{align*}
\mu(Q_{c}) = \frac{\mu_{Q}(\I_{-a})}{\mu_{Q}(\I_{a})} \cdot \mu(Q) \geq \frac{(1-2\beta_a)\nu(\I_{-a})}{(1+2\beta_a)\nu(\I_{a})} \cdot \mu(Q) = \rho_a^{-1} p_a \mu(Q)
\end{align*}
and in similar manner
\begin{align*}
\mu(Q_{c}) \leq \rho_a  p_a  \mu(Q).
\end{align*}
Since $\P(\I_b) = 1$ we have:
$$\rho_a^{-1} p_a \leq \P(A_{a,n}^b) = \sum_{Q \in \cQ^k_a \atop Q \subset \I_b}\P(Q_c) \leq \rho_a p_a.$$
Fix $n,l \in \N$ and estimate the $\P$ measure of the intersection $A_{a,n}^b \cap A_{a,l}^b$. If the generations $k_n = k_l$, which we chose accordingly to $n$ and $l$, the cube unions $A_{a,n}^b = A_{a,l}^b$, and so
$$\P(A_{a,n}^b \cap A_{a,l}^b) = \P(A_{a,n}^b) \leq \rho_a p_a.$$
Suppose $k_n < k_l$. Then for each $Q \in \cQ^{k_n}_a$ we can decompose the central cube $Q_c$ into the generation $k_l$ subcubes:
$$Q_c = \bigcup_{R \in \cQ^{k_l}_a \atop R \subset Q_c} R.$$
In particular, the intersecting cubes
$$A_{a,n}^b \cap A_{a,l}^b = \bigcup_{Q \in \cQ^{k_n}_a \atop Q \subset \I_b} \bigcup_{R \in \cQ^{k_l}_a \atop Q \subset Q_c} R_c,$$
see Figure \ref{figIntersection} for an illustration.

\begin{figure}[h]
\begin{center}
\includegraphics[scale=0.65]{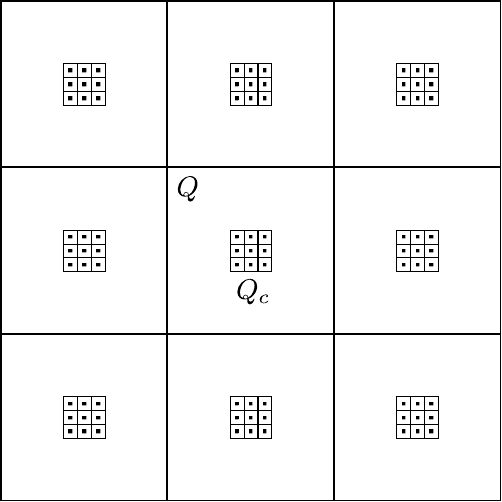}
\end{center}
\caption{Illustration of the intersection $A_{a,n}^b \cap A_{a,l}^b$. Estimating the mass of this intersection is then reduced to estimating the ratios between small black cubes $R_c$ and $Q$. This comparison produces an error given by the number $\rho_a$.}
\label{figIntersection}
\end{figure}

Invoking $\P(\I_b) = 1$ we can now estimate:
\begin{align*}
\P(A_{a,n}^b \cap A_{a,l}^b) =  \sum_{Q \in \cQ^{k_n}_a \atop Q \subset \I_b} \sum_{R \in \cQ^{k_l}_a \atop R \subset Q_c} \P(R_c) \leq  \sum_{Q \in \cQ^{k_n}_a \atop Q \subset \I_b} \rho_a p_a \P(Q_c) \leq \rho_a^2p_a^2.
\end{align*}
Similarly, if the generations $k_n > k_l$ we have the same result, since we can just change the order of $n$ and $l$. Fix $N \in \N$. Then by the estimates above
$$\sum_{n = 1}^N \sum_{l = 1}^N \P(A_{a,n}^b \cap A_{a,l}^b) \leq N[(N-1)\rho_a^2p_a^2+\rho_a p_a]$$
and
$$\Big(\sum_{n = 1}^N \P(A_{a,n}^b)\Big)^2 \geq N^2\rho_a^{-2}p_a^2.$$
On the other hand, the sum
$$\sum_{n = 1}^\infty \P(A_{a,n}^b) \geq \sum_{n = 1}^\infty \rho_a^{-1}p_a = +\infty,$$
since $\rho_a^{-1}p_a > 0$ is a number independent of $n$. So we are allowed to apply Lemma \ref{borelcantelli}:
\begin{align*}\P(A_a^b)&\geq \limsup_{N \to \infty} \frac{\Big(\sum_{n = 1}^N \P(A_{a,n}^b)\Big)^2}{\sum_{n = 1}^N \sum_{l = 1}^N \P(A_{a,n}^b \cap A_{a,l}^b)} \\
& \geq \limsup_{N \to \infty} \frac{N^2 \rho_a^{-2}p_a^2}{N[(N-1)\rho_a^2p_a^2+\rho_a p_a]} = \rho_a^{-4},\end{align*}
which is exactly \eqref{downmeasure}. 

We are now practically finished, since \eqref{downmeasure} implies for any $a \geq b$ that
$$\P\big(\bigcup_{a' \geq a} A^b_{a'}\big) \geq \P(A_a^b) \geq \rho_a^{-4},$$
so by the convergence of measures and the fact that $\rho_a \searrow 1$, as $a \to \infty$, we obtain:
$$\P(A^b) = \P\big(\bigcap_{a \in \N}\bigcup_{a' \geq a} A^b_{a'}\big)= 1.$$
Then recalling that $\P = \mu(\I_b)^{-1} \mu \llcorner \I_b$, we have shown:
$$\mu(\R^d \setminus A) \leq \sum_{b = b_0}^\infty \mu(\I_b \setminus A^b) = 0,$$
so $\mu$ almost every $x \in \R^d$ is an element of $A$.

Lemma \ref{tangentsforR} is thus proved if we can show that $\nu$ is a tangent measure of $\mu$ at every $x \in A$. Fix an $x \in A$, and choose infinitely many $a \in \N$ such that $x \in A_a$. Fix such an $a$ and choose infinitely many $n \in \N$ such that $x \in Q_c$ for the unique $Q \in \cQ^k_a$ for which $x \in Q$. Recall the estimate \eqref{distbound} from the beginning of the proof, that is, the choice of $k = k_n$ implies that each of these cubes have the property
$$F_{a+1}(\mu_Q,\nu_{a}^\w) < \epsilon_{a}\epsilon_{a}^{\w},$$
for some constant $c = c(Q) > 0$ and weights $\w = \w(Q) \in \W$, where 
$$\mu_Q = \cT_{x(Q),r_a^k,c}(\mu) = cT_{x(Q),r_a^k\sharp} \mu.$$ 
Then after passing to a subsequence, we may find increasing sequences $(a_i)_{i \in \N}$ and $(k_i)_{i \in \N}$ of natural numbers such that $a_i,k_i \nearrow \infty$ and for any $i \in \N$ we have:
\begin{itemize}
\item[(1)] the point $x \in Q_{i,c}$, where $Q_{i,c}$ is the central cube of $Q_i$ and $Q_i$ is the unique cube in $\cQ^{k_i}_{a_i}$ containing $x$;
\item[(2)] if $x_i = x(Q_i)$ and $r_i = r_{a_i}^{k_i} = 3^{-(k_i+1)a_i}/2 \searrow 0$, then the distance
$$F_{a_i+1}(c_iT_{x_i,r_i \sharp} \mu,\nu_{a_i}^{\w_i}) < \epsilon_{a_i}\epsilon_{a_i}^{\w_i}$$ 
for some weights $\w_i \in \W$ and constants $c_i > 0$.
\end{itemize}
We will now show that 
$$c_j T_{x,r_j \sharp} \mu \to \nu, \quad \text{as } j \to \infty.$$
By Remark \ref{remarkmetric}(1) it is enough to verify $F_{b}(c_j T_{x,r_j \sharp} \mu,\nu) \to 0$ as $j \to \infty$ for any fixed $b \in \N$. Let $b \in \N$ and $\phi \in \cL(b)$. After passing to a subsequence, we may assume that $\I_{b+1} \subset \I_{a_i+1}$ for all $i \in \N$. Write $z_i := (x-x_i)/r_i$, $i \in \N$. Since 
$$|z_i| = \frac{|x-x_i|}{3^{-a_i}\ell(Q_i)/2} \leq \frac{\ell(Q_{i,c})}{3^{-a_i}\ell(Q_i)/2} = 2 \cdot 3^{-{a_i}},$$
this particularly implies that $\spt (\phi \circ T_{z_i,1}) = \spt \phi + z_{i} \subset \I_{b+1} \subset \I_{a_i+1}$ for every $i \in \N$. On the other hand by the definition of $\nu^{\w_i}_{a_i}$ we have $\nu_{a_i}^{\w_i} \llcorner \I_{a_i} = \nu_{a_i}$ and so:
$$\int \phi\, dT_{z_i,1\sharp}\nu_{a_i}^{\w_i} = \int \phi\, dT_{z_j,1\sharp}\nu =\int_{\I_{b+1}} \phi(x-z_i) \, d\nu x.$$
Hence using the fact that $\phi$ is $1$-Lipschitz we have shown:
\begin{align*}
&\left|\int \phi\, dT_{z_i,1\sharp}\nu_{a_i}^\w  - \int \phi \, d\nu\right| = \left|\int_{\I_{b+1}} \phi(x-z_i) \, d\nu x  - \int_{\I_{b+1}} \phi \, d\nu\right| \\
 & \leq \int_{\I_{b+1}} |\phi(x-z_i) - \phi(x)| \, d\nu x \leq |z_i|\nu(\I_{b+1}) \leq 2 \cdot 3^{-a_i} \nu(\I_{b+1}).
\end{align*}
The mapping $\phi \circ T_{z_i,1} \in \cL(a_i+1)$: we already had $\spt (\phi \circ T_{z_i,1}) \subset \I_{a_i+1}$, and it is $1$-Lipschitz:
$$|\phi \circ T_{z_i,1}(y) - \phi \circ T_{z_i,1}(z)| \leq |T_{z_i,1}(y)-T_{z_i,1}(z)| = |y-z|, \quad y,z \in \R^d.$$
Hence as Definition \ref{choices} in particularly gives $\epsilon_{a_i}^{|\w_i|} < \epsilon_{a_i}$, we have:
\begin{align*}
& \Big|\int \phi \, d(c_i T_{x,r_i \sharp} \mu)- \int \phi \, dT_{z_i,1\sharp}\nu_{a_i}^{\w_i}\Big| \\
&= \Big|\int \phi \circ T_{z_i,1} \,  d(c_i T_{x_i,r_i \sharp} \mu)- \int\phi \circ T_{z_i,1} \, d\nu_{a_i}^{\w_i}\Big|\\
& \leq F_{a_i+1}(c_i T_{x_i,r_i \sharp} \mu,\nu_{a_i}^{\w_i}) < \epsilon_{a_i}\epsilon_{a_i}^{\w_i} < \epsilon_{a_i} \cdot \epsilon_{a_i}.
\end{align*}
Since $\phi \in \cL(b)$ is arbitrary, we have reached our goal:
\begin{align*}
F_b(c_i T_{x,r_i \sharp} \mu,\nu) \leq \epsilon_{a_i} \cdot \epsilon_{a_i} + 2 \cdot 3^{-a_i} \nu(\I_{b+1}) \longrightarrow 0,
\end{align*}
as $i \to \infty$, finishing the proof of Lemma \ref{tangentsforR}.
\end{proof}

Combining Lemma \ref{opendense} and Lemma \ref{tangentsforR}, we have shown that a typical measure $\mu \in \cM$ satisfies $\Tan(\mu,x) = \cM\setminus \{0\}$ at $\mu$ almost every $x \in \R^d$, and thus Theorem \ref{main1} is proven.

\section{Measures are typically non-doubling}
\label{application}

As a direct consequence of Theorem \ref{main1} we can say something about the doubling behavior of typical measures.

\begin{definition}
A measure $\mu \in \cM$ satisfies the \textit{doubling condition} at $x \in \R^d$ if
$$\limsup_{r \to 0} \frac{\mu(B(x,2r))}{\mu(B(x,r))} < \infty.$$
Measure $\mu$ is \textit{non-doubling} if the doubling condition fails at $\mu$ almost every $x \in \R^d$. Notice that non-doubling measures are always singular with respect to the Lebesgue measure in $\R^d$.
\end{definition}

\begin{corollary}
\label{cor1}
A typical measure $\mu \in \cM$ is non-doubling.
\end{corollary}

\begin{proof}
The results of Preiss in \cite[Proposition 2.2 and Corollary 2.7]{Pre87} imply that doubling condition of $\mu$ at $x$ can be characterized by the existence of a constant $C \geq 1$ such that for every $\nu \in \Tan(\mu,x)$ and $r > 0$ we have 
$$\nu(B(0,2r)) \leq C\nu(B(0,r)).$$
If $\Tan(\mu,x) = \cM \setminus \{0\}$, then clearly the doubling condition cannot be satisfied: for example measures $\nu_n = \cL^d \llcorner B(0,n)^c$ satisfy: 
$$\nu_n(B(0,2n)) > 0 = C\nu_n(B(0,n))$$ 
for any $C \geq 1$, yet $\nu_n \in \Tan(\mu,x)$ for every $n \in \N$. Hence the claim follows from Theorem \ref{main1}.
\end{proof}

\begin{remark}
Bate and Speight proved in \cite{BS11} that when a measure $\mu$ on a metric space admits a differentiable structure, then $\mu$ satisfies the doubling condition $\mu$ almost everywhere. Hence Corollary \ref{cor1} also says that with respect to the Euclidean metric a typical $\mu$ in $\R^d$ does not admits a differentiable structure in $\R^d$. It would be interesting to see if Corollary \ref{cor1} could be generalized to other interesting classes of complete metric spaces.
\end{remark}

\section{Sharpness of the result}
\label{sharpness}

A natural question to ask further is that can the property $\Tan(\mu,x) = \cM\setminus\{0\}$ be made to hold at \text{every} point $x \in \spt \mu$ of a typical measure $\mu$. However, this is not possible by the following observation. Here $\cL$ is the Lebesgue-measure on $\R$ and $\cL^+$ is the Heaviside-measure $\cL \llcorner [0,\infty)$.

\begin{proposition}
\label{nontangent}
If $\mu$ is a measure on $\R$ with non-empty support, then there exists $x \in \spt \mu$ such that $\cL \notin \Tan(\mu,x)$ or $\cL^+ \notin \Tan(\mu,x)$.
\end{proposition}

\begin{remark}
Even though the statement of Proposition \ref{nontangent} is in $\R$, it could be extended to $\R^d$ with a similar proof. More precisely, we can use nearly similar techniques to show that for any $\mu \in \cM$ with non-empty support there exists $x \in \spt \mu$ such that either the Lebesgue-measure $\cL^d \notin \Tan(\mu,x)$ or $\cL^{d,+} \notin \Tan(\mu,x)$, where the measure $\cL^{d,+}$ is $\cL^d$ restricted to the set 
$$\R^{d,+} = \{(x_1,x_2,\dots,x_d) : x_i \geq 0 \text{ for all }i = 1,\dots,d\}.$$
\end{remark}

Before we state the proof, let us first observe the following

\begin{remark}
\label{heavisidetangent}
If $\mu$ is a measure, $x \in \spt \mu$, and for some $c_i > 0$ and $r_i \searrow 0$, we would have $c_i T_{x,r_i \sharp} \mu \to \cL^+$, then we can choose a subsequence $(i_j)_{j \in \N}$ such that 
$$\mu(B(x,r_{i_j}))^{-1} T_{x,r_{i_j} \sharp} \mu \to \cL^+, \quad \text{as }j \to \infty.$$ 
This is verified in \cite[Remark 14.4(1)]{Mat95} if we use it in the case $R = 1$ and $\nu = \cL^+$, and notice that $\cL^+(U(0,1)) = \cL^+(B(0,1)) = 1$. 
\end{remark}

\begin{proof}[Proof of Proposition \ref{nontangent}]
Write $A = \spt \mu$. We have two separate cases.

$1^\circ$ Suppose $A$ is a proper subset of $\R$. Since $A$ is closed and non-empty, we can choose $x \in A$ and $\epsilon > 0$ such that either $(x-\epsilon,x) \cap A = \emptyset$ or $(x,x+\epsilon) \cap A = \emptyset$. Let us prove that 
$$\cL \notin \Tan(\mu,x).$$
We may assume that $(x,x+\epsilon) \cap A = \emptyset$, the other case is symmetric. Suppose on the contrary that there exists $c_i > 0$ and $r_i \searrow 0$ such that $c_i T_{x,r_i\sharp} \mu \to \cL$ as $i \to \infty$. Fix $i_0 \in \N$ such that $r_i < \epsilon$ for each $i \geq i_0$. Fix a continuous $\phi : \R \to \R$ such that $\spt \phi \subset (0,1)$ and $\int \phi \, d\cL = 1$. Then for each $i \geq i_0$ we have:
$$\int \phi \, d\cL = 1 \neq 0 = \int \phi \, d(c_i T_{x,r_i \sharp} \mu),$$
which is a contradiction with $c_i T_{x,r_i\sharp} \mu \to \cL$. Hence $\cL \notin \Tan(\mu,x)$.

$2^\circ$ Suppose $A = \R$. Let us now find $x \in \R$ such that
$$\cL^+ \notin \Tan(\mu,x).$$
If $x \in \R$ and $r > 0$, denote $c_{x,r} = \mu(B(x,r))^{-1}$. Fix any number $0 < \epsilon < 1/20$. Then the constant $\epsilon' := \epsilon/16 - 5\epsilon^2/4 > 0$. Fix any $y_0 \in \R$. Pick some $r_0 > 0$ such that 
$$F_3(c_{y_0,r_0} T_{y_0,r_0\sharp} \mu,\cL^+) < \epsilon^2,$$
recall the definition of $F_3$ in Definition \ref{metricmeasure}. If we cannot choose such $r_0$, Remarks \ref{remarkmetric}(1) and \ref{heavisidetangent} would imply that $\cL^+ \notin \Tan(\mu,y_0)$, which finishes the proof. Write $r_i = 4^{-i}r_0$, $i \in \N$. Let us now construct a sequence of points $x_0,x_1,x_2, \dots \in \R$. First we let $x_0 = y_0 + r_0$. Fix $i \geq 1$ and suppose the points $x_0,\dots,x_{i-1}$ have already been constructed. If there exists $y_i \in [x_{i-1},x_{i-1}+r_i]$ and $s_i \in (r_{i+1},r_i]$ such that
\begin{align}
\label{yiest}
F_3(c_{y_i,s_i} T_{y_i,s_i\sharp} \mu,\cL^+) < \epsilon^2,
\end{align}
we let $x_i = y_i+r_i$, see Figure \ref{figPointy}. Otherwise, if such a choice cannot be made, we let $x_i = x_{i-1}$.

\begin{figure}[h]
\begin{center}
\includegraphics[scale=0.8]{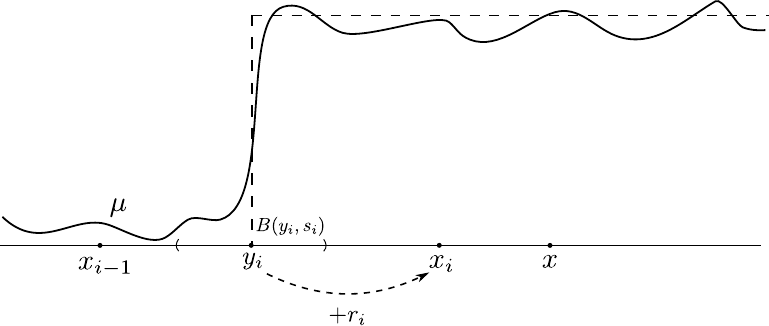}
\end{center}
\caption{The choice of the point $x_i$ when we can choose $y_i \in [x_{i-1},x_{i-1}+r_i]$ and $s_i \in (r_{i+1},r_i]$ such that \eqref{yiest} is satisfied. Since $c_{y_i,s_i} T_{y_i,s_i\sharp} \mu$ is close to the Heaviside-measure $\cL^+$, we move to the right in the construction since here $c_{x_i,r} T_{x_i,r\sharp}\mu$ is quite far from $\cL^+$ with respect to the distance $F_3$ for all scales $r \in (r_{i+1},r_i]$ by the choice of $s_i$. Furthermore, the limit $x = \lim_{i \to \infty} x_i$ is then quite close to the point $x_i$ and thus also $c_{x,r} T_{x,r\sharp}\mu$ is quite far from $\cL^+$ for all scales $r \in (r_{i+1},r_i]$.}
\label{figPointy}
\end{figure}

This way we have constructed a sequence of reals $x_0,x_1,x_2,\dots$ such that for any $i \geq 1$ either we could choose $y_i \in [x_{i-1},x_{i-1}+r_i]$ and $s_i \in (r_{i+1},r_i]$ such that \eqref{yiest} is satisfied and $x_i := y_i + r_i$, \textit{or} $x_i := x_{i-1}$, whence
\begin{align}
\label{lowery}
F_3(c_{y,s} T_{y,s\sharp} \mu,\cL^+) \geq \epsilon^2 \quad \text{ for all } y \in [x_{i-1},x_{i-1}+r_i] \text{ and } s \in (r_{i+1},r_i].
\end{align}
Fix now $i,j \in \N$, $j > i$. By construction for any $k \in \N$ we have $x_k \in [x_{k-1},x_{k-1}+2r_k]$, so
$$x_j \in [x_i,x_i+2r_{i+1}+2r_{i+2}+\dots+2r_j] \subset [x_i,x_i+r_i],$$
since $\sum_{i = 1}^\infty 4^{-i} = 1/3 < 1/2$. Thus the limit $x = \lim_{i \to \infty} x_i$ exists and $x \in [x_i,x_i+r_i]$ for every $i \in \N$.

Fix a radius $0 < r \leq r_0$ and choose $i \in \N$ such that $r_{i+1} < r \leq r_i$. Suppose \eqref{yiest} holds. Define a map $H_i : \R \to \R$ by:
$$H_i(z) = \frac{s_i}{r}z+T_{x,r}(y_i), \quad z \in \R.$$ 
Then by definition $H_i \circ T_{y_i,s_i} = T_{x,r}$. Let $\phi : \R \to [0,\epsilon]$ be any Lipschitz-map with
$$\Lip \phi \leq 1/4, \quad \spt \phi \subset [-1,0] \quad \text{and} \quad \phi|[-1+\epsilon,-\epsilon] = \epsilon/4;$$ 
and let $\psi : \R \to [0,\epsilon]$ be any Lipschitz-map with
$$\Lip \psi \leq 1/4, \quad \spt \psi \subset [-1-\epsilon,1+\epsilon] \quad \text{and} \quad \psi|[-1,1] = \epsilon/4.$$
Now in particular $\spt (\phi \circ H_i), \spt (\psi\circ H_i) \subset [-12,12] \subset \I_3$ and 
$$\Lip (\phi \circ H_i), \Lip (\psi\circ H_i) \leq \frac{s_i}{r} \cdot \frac{1}{4} < 1$$
since $s_i \leq r_i$ and $r > r_{i+1} = r_i/4$. Hence $\phi \circ H_i,\psi \circ H_i \in \cL(3)$, so by \eqref{yiest} we have:
$$c_{y_i,s_i} \int \psi \circ H_i \, dT_{y_i,s_i\sharp} \mu \leq \int \psi \circ H_i \, d\cL^+ + \epsilon^2.$$
Hence
\begin{align*}
\mu(B(x,r)) &\leq \frac{4}{\epsilon}\int \psi \, dT_{x,r\sharp} \mu = \frac{4\mu(B(y_i,s_i))}{\epsilon} \cdot c_{y_i,s_i} \int \psi \circ H_i \, dT_{y_i,s_i\sharp} \mu \\
& \leq \frac{4\mu(B(y_i,s_i))}{\epsilon} \cdot \Big(\int \psi \circ H_i \, d\cL^+ + \epsilon^2\Big) \\
& \leq \frac{4\mu(B(y_i,s_i))}{\epsilon}(\epsilon\cL^+([-12,12])+\epsilon^2)\leq 52\mu(B(y_i,r_i)).
\end{align*}
If $I = H_i^{-1}[-1+\epsilon,-\epsilon]$, then $I$ is an interval and of the length
$$\ell(I) = (1-2\epsilon)r/s_i \geq 1/4-\epsilon.$$
Moreover, the choice of $\phi$ yields $(\phi \circ H_i)|I = \epsilon/4$. Thus by \eqref{yiest} we have
$$\int \phi \circ H_i \, d\cL^+ - \epsilon^2 \leq c_{y_i,s_i} \int \phi \circ H_i \, dT_{y_i,s_i\sharp} \mu.$$
Hence
\begin{align*}
&\Big|\int \phi \,d(c_{x,r} T_{x,r \sharp} \mu)-\int \phi \, d\cL^+\Big| = c_{x,r}\int \phi \,dT_{x,r \sharp} \mu = c_{x,r} \int \phi \circ H_i\,dT_{y_i,s_i \sharp} \mu \\
& \geq \frac{c_{x,r}}{c_{y_i,s_i}} \Big(\int \phi \circ H_i\,d\cL^+ - \epsilon^2\Big) \geq \frac{c_{x,r}}{c_{y_i,s_i}} \Big(\frac{\epsilon}{4}\cL^+(I) - \epsilon^2\Big) \geq \frac{c_{x,r}}{c_{y_i,s_i}} \epsilon' \geq \epsilon'/52.
\end{align*}
Therefore,
$$F_{2}(c_{x,r} T_{x,r \sharp} \mu,\cL^+) \geq \epsilon'/52.$$
On the other hand, if \eqref{lowery} holds for the index $i \in \N$, we immediately have:
$$F_{2}(c_{x,r} T_{x,r \sharp} \mu,\cL^+) \geq \epsilon^2$$
since $x \in [x_i,x_i+r_i] = [x_{i-1},x_{i-1}+r_i]$ in this case. Hence for any $0 < r \leq r_0$ we have: 
$$F_{2}(c_{x,r} T_{x,r \sharp} \mu,\cL^+) \geq \min\{\epsilon'/52,\epsilon^2\},$$
yielding that $\cL^+ \notin \Tan(\mu,x)$ by Remarks \ref{remarkmetric}(1) and \ref{heavisidetangent}.
\end{proof}

\section{Micromeasures}
\label{secmicro}

The notion of micromeasures is a symbolic way to define local blow-ups of measures in trees, and in this setting we can also obtain a similar result to Theorem \ref{main1}. Micromeasures have just recently been considered for instance in \cite{HS09, Shm11}. Let $I = \{1,2,\dots,b\}$, where $b \in \N$ is fixed. If $n \in \N$, we write
$$I^n = \{(x_1,x_2,\dots,x_n) : x_i \in I\},\,\, I^* = \bigcup_{n \in \N} I^n \,\, \text{and} \,\, I^\N = \{(x_1,x_2,\dots) : x_i \in I\}.$$
Then $I^\N$ is a compact metric space with the metric $d(x,y) = 2^{-(x \land y)}$, $x,y \in I^\N$, where $x \land y$ is the first index $i \in \N$ when $x_i$ differs from $y_i$. When $x = (x_1,x_2,\dots) \in I^\N$ or $x \in I^m$ with $m \geq n$, we let $x|n \in I^n$ be the $n$:th cut of $x$, that is $x|n = (x_1,x_2,\dots,x_n)$. If $y \in I^n$, we let the cylinder generated by $y$ be:
$$[y] := \{x \in I^\N: x_i = y_i, i = 1,\dots,n\}.$$
Let $\cP = \cP(I^\N)$ be the set of all Borel probability measures on $I^\N$. If $\mu \in \cP$ and $y \in I^n$ with $\mu[y] > 0$, we denote 
$$\mu_y[z] = \frac{\mu[yz]}{\mu[y]}, \quad z \in I^*,$$
that is, the normalized restriction of $\mu$ to $[y]$ shifted back to $I^\N$. This notion defines a Borel probability measure on $I^\N$. We can metrize $\cP$ with the following distance:
$$\pi(\mu,\nu) = \sup_{\phi \in \cL}\Big|\int \phi \, d\mu - \int \phi \, d\nu\Big|, \quad \mu,\nu \in \cP,$$
where $\cL$ is the set of all Lipschitz-maps $\phi : I^\N \to \R$ with $\Lip \phi \leq 1$ and maximal value $\|\phi\|_{\infty} \leq 1$. The set $\cP$ can be equipped with the weak topology, which agrees with the topology induced by $\pi$. Moreover, the compactness of $I^\N$ yields that $(\cP,\pi)$ is a compact metric space.

\begin{definition}
\label{micro}
A probability measure $\nu \in \cP$ is a \textit{micromeasure} of $\mu \in \cP$ at $x \in I^\N$ if there exists $n_i \nearrow \infty$ such that
$$\mu_{x|n_i} \longrightarrow \nu, \quad \text{as }i \to \infty.$$ 
The set of micromeasures of $\mu$ at $x$ is denoted by $\micro(\mu,x)$, which is a closed subset of $\cP$.
\end{definition}

We obtain the following

\begin{theorem}
\label{microtypical}
A typical $\mu \in \cP$ satisfies $\micro(\mu,x) = \cP$ at every $x \in I^\N$.
\end{theorem}

\begin{proof} The proof below resembles the proof of Theorem \ref{main1}, but the steps are dramatically simpler. Namely, here we do not have to worry about the measures of boundaries nor how to fit balls and cubes to each other and the same time worry about $\mu$ almost every point. The core is similar, so we will leave out some of the details. 

First of all, choose a countable dense $\cS \subset \cP$ such that $\nu[y] > 0$ for every $\nu \in \cS$ and $y \in I^*$. When $y \in I^*$ and $\mu[y] > 0$ we denote:
$$\cT_y(\mu) = \mu_y.$$
With this in mind, define
$$\cR = \bigcap_{\nu \in \cS} \bigcap_{n \in \N} \cR_{\nu,n} \quad \text{and} \quad \cR_{\nu,n} = \bigcup_{k \geq n} \bigcap_{y \in I^k} \cT_y^{-1} U(\nu,1/k),$$
where the ball $U(\nu,1/n)$ is taken with respect to the metric $\pi$. Suppose $\nu \in \cS$ and $n \in \N$. Let us first verify that $\cR_{\nu,n}$ is open and dense in $\cP$. \begin{itemize}
\item[(1)] Since $\partial[y] = \emptyset$ for any $y \in I^*$, the map $\cT_y : \{\mu \in \cP :\mu[y] > 0\} \to \cP$ is continuous. Moreover, the set $\{\mu \in \cP : \mu[y] > 0\} \subset \cP$ is open, which yields that for any open $U \subset \cP$ the pre-image $\cT_y^{-1}U$ is open in $\cP$. In particular, $\cR_{\nu,n}$ is open in $\cP$.
\item[(2)] If $\mu' \in \cP$ and $\epsilon > 0$, we may choose $\mu \in \cP$ such that $\mu[y] > 0$ for every $y \in I^*$ and $\pi(\mu,\mu') < \epsilon/2$. Fix $k \in \N$ and denote
$$\mu^k = \sum_{y \in I^k} \mu[y] \nu^y,$$
where $\nu^y[z] = \nu[z]/\nu[y]$ for each $z \in I^m$, $m \geq k$, with $z|k = y$ and $\nu^y[z] = 0$ otherwise. Then $\cT_y(\mu^k) = \nu$ for each $y \in I^k$, $k \in \N$, so $\mu^k \in \cR_{\nu,n}$ if $k \geq n$. Moreover, as in the proof of Lemma \ref{opendense}, we have: $\mu_k \to \mu$, as $k \to \infty$ (we just replace $Q$ by $y$). Thus we can fix $k \geq n$ such that $\pi(\mu^k,\mu) < \epsilon/2$ yielding $\pi(\mu^k,\mu') < \epsilon$. In particular, $\cR_{\nu,n}$ is dense in $\cP$.
\end{itemize}
In order to finish the proof we fix $\mu \in \cR$ and verify that for a fixed $x \in I^\N$ we have $\micro(\mu,x) = \cP$. Since $\micro(\mu,x)$ is closed in $\cP$, and $\cS \subset \cP$ is dense, we only need to check that $\nu \in \micro(\mu,x)$ for a fixed $\nu \in \cS$. By the definition of $\cR$, there exists $n_i \nearrow \infty$, $i \to \infty$, such that
$$\mu \in \bigcap_{y \in I^{n_i}} \cT_y^{-1} U(\nu,1/n_i), \quad i \in \N.$$
Especially $\mu \in \cT_{x|n_i}^{-1} U(\nu,1/n_i)$ for any $i \in \N$. This is exactly what we wanted:
$$\pi(\mu_{x|n_i},\nu) < 1/n_i, \quad i \in \N,$$
so $\mu_{x|n_i} \to \nu$ as $i \to \infty$.
\end{proof}

\section{Further problems}
\label{problems}

\subsection{Micromeasure distributions} Micromeasure distributions provide a probabilistic way to describe which measures tend to occur more often as local blow-ups $\mu_{x|n}$ of $\mu \in \cP$, and thus tell us what the ``expected'' micromeasures of $\mu$ are. Let us first expand some of the notation in Section \ref{micro}. This notation was used in \cite{Shm11}. Write
$$\Xi = \{(\mu,x) \in \cP(I^\N) \times I^\N : \mu[x|n] > 0 \text{ for all }n \in \N\}.$$
Let $\sigma : I^\N \to I^\N$ be the shift, that is, $\sigma(x_1,x_2,\dots) = (x_2,x_3,\dots)$, if $x = (x_1,x_2,\dots) \in I^\N$. Define the map $\ZOOM : \Xi \to \Xi$ by
$$\ZOOM(\mu,x) = (\mu_{x_1},\sigma x), \quad (\mu,x) \in \Xi.$$
If $n \in \N$ let $\ZOOM^n$ be the $n$-fold composition of the mapping $\ZOOM$. Notice that by definition $\ZOOM^n(\mu,x) = (\mu_{x|n},\sigma^n x)$.

\begin{definition}
\label{microdistribution}
Fix $(\mu,x) \in \Xi$, that is, $\mu \in \cP$ and $x \in \spt \mu$. We say that a Borel probability measure $P$ on $\Xi$ is a \textit{micromeasure distribution} of $\mu$ at $x$ if there exists $N_i \nearrow \infty$, $i \to \infty$, such that
$$\frac{1}{N_i}\sum_{n = 1}^{N_i} \delta_{\ZOOM^n(\mu,x)} \longrightarrow P, \quad \text{as } i \to \infty,$$
where the convergence is taken with respect to the weak topology on $\cP(\Xi)$. 
\end{definition}

We already know that any measure is a micromeasure of a typical measure $\mu \in \cP$, but could we say something more about their distribution?

\begin{problem}
\label{problemmicrodist}
What are the micromeasure distributions of a typical measure $\mu \in \cP$?
\end{problem}

Similarly, one could ask an analogous question for \textit{tangent measure distributions}, see for example \cite{MP98}.

\subsection{Prevalence} Prevalence is a notion of genericity that was originally motivated by the need to have a ``translation-invariant'' measure theoretical form of genericity in infinite dimensional vector spaces. The natural finite dimensional analogue of it could be the notion of ``Lebesgue almost every'' in $\R^d$. The ideas surrounding prevalence were introduced by Christensen in \cite{Chr72, Chr74}, and the name ``prevalence'' was suggested by Hunt, Sauer, and Yorke in \cite{HSY92}. The notion of prevalence was originally only defined for elements in a topological vector space, but in \cite{AZ01} Anderson and Zame also gave an analogous definition for convex subsets of topological vector spaces.

\begin{definition} Let $X$ be a topological vector space and let $C$ be a completely metrizable convex subset of $X$. We say that a set $E \subset C$ is \textit{shy in $C$ at the point $c \in C$} if for every $\delta \in (0,1)$, and open neighbourhood $U$ of the origin in $X$ there exists a Borel measure $\Lambda$ on $X$ with $\Lambda(X) > 0$ such that
\begin{itemize}
\item[(1)] $\spt \Lambda$ is compact, $\spt \Lambda \subset U+c$, and $\spt \Lambda \subset \delta C + (1-\delta)c$;
\item[(2)] $\Lambda(x+E) = 0$ for every $x \in X$.
\end{itemize}
If $E$ is shy in $C$ at every point $c \in C$, then we say $E$ is \textit{shy in $C$}. A property $\sf{P}$ of points in $x \in C$ is satisfied for \textit{prevalent} $x \in C$ if the set
$$\{x \in C : x \text{ does not satisfy } \sf{P}\}$$
is shy in $C$.
\end{definition}

In our case we could consider the set $\cP(K)$ of all Borel probability measures on $K$, where $K$ is some compact subset of $\R^d$, and $\overline{\cM}(K)$ the set of all \textit{signed} Borel measures on $K$. Then $\cP(K)$ is a completely metrizable convex subset of the topological vector space $\overline{\cM}(K)$. This setting was already considered by Olsen in \cite{Ols10} when the $L^q$-dimension of prevalent measures $\mu \in \cP(K)$ was studied. Moreover, in the case of trees $I^\N$, the set $\cP(I^\N)$ is a complete convex subset of $\overline{\cM}(I^\N)$, the set of signed Borel measures on $I^\N$.

\begin{problem}
\label{problemprevalence}
What are the tangent measures of prevalent measures in $\cP(K)$? What are the micromeasures and micromeasure distributions of prevalent measures in $\cP(I^\N)$?
\end{problem}

\medskip
\noindent {\bf Acknowledgment}. The author thanks the anonymous referee of \cite{OS12} for recommending to study the tangent measures of typical measures. Moreover, the author is grateful to Pertti Mattila and Tuomas Orponen for valuable discussions and comments, and especially to Tuomas for suggesting the idea for Proposition \ref{nontangent}.

\end{document}